\documentclass[12pt]{amsart}

\usepackage{latexsym, amssymb}
\usepackage[usenames]{color}

\newcommand{\be}{\begin{equation}}
\newcommand{\ee}{\end{equation}}
\newcommand{\beq}{\begin{eqnarray}}
\newcommand{\eeq}{\end{eqnarray}}

\newtheorem{prop}{Proposition}[section]
\newtheorem{theo}[prop]{Theorem}
\newtheorem{lemm}[prop]{Lemma}
\newtheorem{coro}[prop]{Corollary}
\newtheorem{rema}[prop]{Remark}

\newtheorem{defi}[prop]{Definition}

\def\begeq{\begin{equation}}
\def\endeq{\end{equation}}
\def\p{\partial}

\begin{document}

\title { Exhaustion of isoperimetric regions  in asymptotically hyperbolic manifolds with  scalar curvature $R\geq -6$}

\begin{abstract}
In this paper, aimed at exploring the fundamental properties of isoperimetric region in $3$-manifold $(M^3,g)$ which is asymptotic
to Anti-de Sitter-Schwarzschild manifold with scalar curvature $R\geq -6$, we prove that connected isoperimetric region $\{D_i\}$ with $\mathcal{H}_g ^3(D_i)\geq \delta_0>0$ cannot slide off to the infinity of $(M^3,g)$ provided that $(M^3,g)$ is not isometric to the hyperbolic space. Furthermore, we prove that isoperimetric region $\{D_i\}$ with topological sphere $\{\partial D_i\}$ as boundary is exhausting regions of $M$ if Hawking mass $m_H(\partial D_i)$ has uniform bound. In the case of exhausting isoperimetric region, we obtain a formula on expansion of  isoperimetric profile in terms of renormalized volume.
\end{abstract}

%\date{October 2009}

\renewcommand{\subjclassname}{\textup{2000} Mathematics Subject Classification}

\author{Dandan Ji $^\dag$, Yuguang Shi$^\ddag$ and Bo Zhu $^\ddag$
 }

\address{Dandan Ji, School of Mathematical Sciences, Capital Normal University, Beijing, 100048.P. R. China.} \email{jidandan@cnu.edu.cn}

\address{Yuguang Shi, Key Laboratory of Pure and Applied mathematics, School of Mathematical Sciences, Peking University,
Beijing, 100871, P.R. China.} \email{ygshi@math.pku.edu.cn}

\address{Bo Zhu, Key Laboratory of Pure and Applied mathematics, School of Mathematical Sciences, Peking University,
Beijing, 100871, P.R. China.}\email{zhubo200916@163.com}

\thanks{$^\dag$Project supported by Beijing Postdoctoral Research Foundation. }
\thanks{$^\ddag$Research partially supported by NSF grant of China  10990013.}

\date{\today}
\maketitle

\markboth{Dandan Ji,  Yuguang Shi and Bo Zhu}{}

\section {introduction}
We first introduce the concept of the asymptotically hyperbolic manifold.
\begin{defi}\label{definitionofahmfd}
A complete, orientated and noncompact Riemannian 3-manifold $(M^3,g)$ is called to be asymptotically hyperbolic (AH) manifold if the following condition is true:
\begin{enumerate}
  \item There exists a compact set $K\subset M$ such that $M\backslash K$ is a union of finite components, and each component, which is called end, is diffeomorphic to $\mathbb{R}^3\backslash B_{r_0}(o)$;
  \item With respect to the spherical coordinates induced by the diffeomorphism above, the metric can be written as
  $$g=dr^2+(\sinh^2r)\cdot \sigma+\frac{1}{3\sinh r}h+O(e^{-3r}).$$
  where $\sigma$ is the standard metric on $\mathbb{S}^2$ and $h$ is a symmetric 2-tensor on $\mathbb{S}^2$. Moreover, the asymptotical expansion can be differentiated any times.
\end{enumerate}
\end{defi}

In many contexts, it is desirable to analyze the geometric quantities of isoperimetric surface (see Definition \ref{isoprofile} for its definition below) at infinity of  AH manifold. For instance, to investigate Penrose inequality for AH manifold (see \cite{BC2}, \cite{WXD}), to study the uniqueness and classifications of isoperimetric surfaces near the infinity of AH manifold (see \cite{CGGK}, \cite{OC}, \cite{NT}). To do that, we need to analyze the behavior of a family of isoperimetric regions $\{D_i\}$ (see Definition \ref{isoprofile} for its definition below) in AH manifold, we are usually faced with the following three situations:  one part of the region drifts off to the infinity, {\it i.e.} eventually it disjoints with any fixed compact domain; always passes through a fixed compact domain; is an exhaustion of the whole mainfold, {\it i.e.}  for any compact domain $K\subset M$, it will be contained in $D_i$ for $i$ large enough. Due to the each connected components of isoperimetric region is still isoperimetric region, we always assume that the isoperimetric regions are connected and it's corresponding boundary isoperimetric surfaces are also connected. In this paper, we give a delicate analysis of the behavior of isoperimetric regions in AH manifold with the scalar curvature $R \geq -6$. Namely, we have

\begin{theo}\label{classification0}
Let $(M^3,g)$ be an AH manifold with the scalar curvature $R(g)\geq -6$ and $h=m\sigma,\ m \in \mathbb{R}$ in Definition \ref{definitionofahmfd}. Suppose that $m>0$ and $\{D_i\}$  is a family of isoperimetric regions with $\mathcal{H}^3_g(D_i)\rightarrow \infty$, we have the following classification:
\begin{enumerate}
  \item $\{D_i\}$ is an exhuastion of $(M,g)$; or
  \item there exists a subsequence of $\{\Sigma_i=\partial D_i\}$ converging to properly, strongly stable(for  definition see Corollary \ref{strongstability}), noncompactly complete hypersurface, each connected component $\mathbb{S}$ of which is a constant mean curvature surface of $H=2$. Furthermore, $\mathbb{S}$ is conformally diffeomorphic to complex plane $\mathbb{C}$.
\end{enumerate}
Here, $\mathcal{H}^3_g(\cdot)$ denotes the Hausdroff measure in $(M, g)$ with respect to the metric $g$.

\end{theo}

Our arguments in Theorem \ref{classification0} above is similar to that developed in asymptotically flat manifold in the paper \cite{MJ2012} and \cite{MJ2013} . Correspondingly, in the Appendix C of \cite{MJ2012},  they obtained the result that the limiting suface in the second case (2) in Theorem \ref{classification0} is area minimizing in asymptotically flat manifold. Furthermore, For the case of nontrivial asymptotically Schwarzschilds manifold, it has been proved that such a limiting surface which is a properly stable minimal surface in that case, cannot exist(see Theorem 1 in \cite{Ca}).  With this fact in mind, we wonder whether the second case in Theorem \ref{classification0} can really happen or not in asymptotically hyperbolic space. Under the additional condition, we can show the following result.

\begin{theo}\label{rigidity0}
Let $\mathbb{S}$ be a connected component of the limiting surface of a family of isoperimetric surfaces $\{\Sigma_i\}$ in an AH manifold of $(M^3,g)$ with $R(g)\geq -6$ and $h=m\sigma,\ m \in \mathbb{R}$ in Definition \ref{definitionofahmfd} , if $\mathbb{S}$ is a noncompact, completely connected surface with $\int_{\mathbb{S}}K\leq0$, then $(M^3,g)$ is isometric to $\mathbb{H}^3$.
\end{theo}

According to \cite{CD}, we see that the assumption of $h = m\sigma$ in Theorem 1.4 is necessary. It is interesting to see that in asymptotically flat case, the limit surface of isoperimetric surfaces is area minimizing which is a much stronger property than stability (see Appendix C in \cite{MJ2013}).With is in mind,  it is natural to conjecture that
a similar property holds for asymptotically hyperbolic manifolds, namely that the limit of large isoperimetric regions in an asymptotically hyperbolic manifold (which one expects to have mean curvature H = 2) minimizes the brane action functional $area(\cdot)-2 vol(\cdot)$ among compactly supported deformations of the surface.  In \cite{CE},   the possibility that large isoperimetric regions always pass through fixed region in asymptotically flat manifold was ruled out by using this property. One might guess that a similar result was true for asymptotically hyperbolic manifolds (not just asymptotically Schwarzschild-anti-de Sitter manifolds, as considered in Theorem \ref{rigidity0}).

\vskip 2mm
As an application of Theorem \ref{rigidity0}, we obtain a criterion for a family of isoperimetric regions being an exhaustion of $(M,g)$:

\begin{theo}\label{exhausting0}
Let $(M^3,g)$ be an AH manifold with $R(g)\geq -6$ and $h=m\sigma$ in Definition \ref{definitionofahmfd} with $m>0$. Suppose that $\{D_i\}$  is a family of isoperimetric regions with $\mathcal{H}^3_g(D_i)\rightarrow \infty$ and $\Sigma_i=\partial D_i$ is topological sphere. If in addition,
$$  m_H(\Sigma_i)\leq C, \quad \text{for all $i$}.$$
Then  $\{D_i\}$ is an exhaustion of $(M,g)$. Here, $m_H(\Sigma)$ is the Hawking mass of $\Sigma$, for its definition see Definition \ref{hawkingmass1}, $C$ is a positive constant independent of $i$.
 \end{theo}

\begin{rema}
It was proved in \cite{OC} that in the case of compact perturbation of Schwarzschild-ADS of positive mass the Hawking mass of an isoperimetric sphere
is uniformly bounded. Also, according to Theorem \ref{exhausting0}, the problem to classify the isoperimetric surfaces with type of topological sphere is reduced to the one that classification of exhausting isoperimetric spheres.
\end{rema}

An interesting notion called renormalized volume  $V(M,g)$ was introduced in \cite{OC}. Namely, let $\Omega_i$ be an domain and an exhaustion of $(M,g)$, we define $V(M, g)=\lim\limits_{i\rightarrow \infty}\left( \mathcal{H}^3_g(\Omega_i) -\mathcal{H}^3_{\mathbb{H}}(\Omega_i) \right)$. Here, $\mathcal{H}^3_{\mathbb{H}}(\Omega_i)$ denotes the volume of domain enclosed by $\p\Omega_i $ in $\mathbb{H}^3$(see Definition 5.1 in \cite{OC}). By the same arguments there, we have $V(M,g)\geq 0$ provided that the scalar curvature of $(M^3,g)$ is at least $-6$ and equality holds iff $(M^3,g)$ is isometric to $\mathbb{H}^3$. It should be interesting to  understand the renormalized volume  $V(M,g)$  in more details. For an exhausting isoperimetric domain $\{D_i\}$, we have

\begin{theo}\label{differav0}
Let $(M^3,g)$ be completely, asymptotically hyperbolic manifold with $R(g)\geq -6$. For any exhausting isoperimetric domain $\{D_i\}$,
we have
$$\lim_{i\rightarrow\infty}(A_g(v_i)-A_{{\mathbb{H}}}(v_i))=-2V(M,g).$$
Here, $A_g(\cdot)$ and $A_{{\mathbb{H}}}(\cdot)$ denote the isoperimetric profiles in $(M^3,g)$ and $\mathbb{H}^3$ respectively (see Definition \ref{isoprofile} for the definition).
\end{theo}

\begin{rema}
More explicit expansion of $A_g(v)$ was obtained in \cite{OC} if  $(M^3,g)$ is a compact perturbation of AdS-Schwarzschild manifold (See equality in P.3 in \cite{OC}).
\end{rema}

\vskip 5mm
Our paper was inspired by \cite{BC} and \cite{OC} and some of these ideas and arguments are from these two interesting papers. One of key steps in the proof of Theorem \ref{classification0} is to prove no drift of a family of isoperimetric regions with positive uniform lower bound volume, we achieve this by the area comparison of isoperimetric surface (see Theorem \ref{areacomparison} below) and make use of renormalized volume (see Theorem \ref{differav0} above and Lemma \ref{largev} below). Another key step in the proof is strong stability of the limiting surface $\mathbb{S}$ (See Corollary \ref{strongstability} below) by which we are able to deduce the conformal type of $\mathbb{S}$ is complex plane, and actually $\mathbb{S}$ is flat if its total curvature is nonpositive. Combining with Lemma \ref{finiteintegoftangentialprho}, we see the total mass of ambient manifold $(M^3, g)$ vanishes. Hence, we get Theorem \ref{rigidity0} by the positive mass theorem proved in \cite{WXD}. The main idea of the proof of Theorem \ref{exhausting0} is to analyze the geometry properties of the limiting surface. Under the  assumption of uniform bound of the Hawking mass of sequence of isoperimetric surface, we deduce that the limiting surface $\mathbb{S}$ is umbilical and total curvature is zero. In addition, by Theorem \ref{rigidity0}, we see such case can only appear when the ambient manifold $(M^3, g)$ is isometric to $\mathbb{H}^3$. In fact, nonexistence of  such limiting surface was proved in \cite{OC} when $(M^3, g)$ is compact perturbation of Ads-Schwarzchild manifold, and arguments there rely heavily on the fact that the ambient manifold is AdS-Schwarzschild manifold outside a compact set (see Lemma 8.1 in \cite{OC}). Finally, we would point out that some arguments in this paper are from \cite{NT}.
\vskip 5mm
The remaining part of the paper goes follows. In Section 2 we prove some basic facts of isoperimetric surfaces used later; in Section 3 we prove an area comparison theorem for isoperimetric surfaces by which we can show all isoperimetric regions with positive uniform lower bounded can not slide off to the infinity provided that the ambient manifold is not isometric to $\mathbb{H}^3$; in Section 4, we prove our main results. Here, we make appointment that the constant $C$ in this paper might be different line by line.
\vskip 1cm
\textbf{Acknowledgments:} we are grateful to Dr. Otis Chodosh and Prof. Xiang Ma for giving some helpful suggestions, Besides, the third author would thank Dr. Zhichao Wang for his discussions in BICMR.  Finally, we would like to thank the referees very much for pointing out many inaccuracies. His or her valuable suggestions make the paper clearer and much more readable.

\vskip 1cm
\section {Preliminary}

One of basic facts for AH manifold is the existence of essential set on each end,

\begin{defi}\label{essentialset}
A compact subset $\mathbb{D}$ of $(M^3,g)$ is called an
essential set if
\begin{enumerate}
  \item $\mathbb{D}$ is a compact domain of $M$ with smooth and convex boundary $\mathbb{B}$ $:= \partial \mathbb{D},$ i.e. its second fundamental form with respect to the outward unit normal vector field is positive definite. Any geodesic half line emitting from $\mathbb{B}$ orthogonally to the outside of $\mathbb{D}$ can be extended to the infinity of $(M,g)$;
  \item The distant function $\rho$ to $\mathbb{D}$ is a smooth function;
\end{enumerate}
\end{defi}

We should notice that if $\mathbb{D}$ is an essential set, then $\rho$ is a smooth function
and has no critical point which implies that the region in $M$
which lies in the outside of the essential set $\mathbb{D}$ is diffeomorphic
to $[0,\infty)\times \mathbb{B}$. Furthermore, we denote $\mathbb{D}_\rho :=\{x\in M: d(x, \mathbb{D})\leq \rho\}$ where $d(\ ,\ )$ is the distant function
with respect to metric $g$. Then $Area (\mathbb{D}_\rho)$ increases along $\rho$ and the surface $\partial \mathbb{D}$ has positive mean curvature with respect to the outward unit normal vector.
\vskip 5mm

In this section, we will prove some basic properties of isoperimetric surfaces in AH manifold  $(M^3,g)$ with the scalar curvature $R\geq -6$.
The first one is on the growth of its area, more specifically, we have

\begin{lemm}\label{areagrowthisosurf}
Let  $(M^3,g)$ be an AH manifold and $\Sigma$ be an isoperimetric surface in $(M^3,g)$. Then, there exists a constant $\Lambda$ depending only on $(M^3,g)$ with
\begin{equation}\label{areagrowthisosurf1}
Area (\Sigma\cap \mathbb{D}_\rho )\leq \Lambda e^{2\rho}.
\end{equation}
\end{lemm}

\begin{proof}
Let $\mathcal{D}$ be an isoperimetric region with boundary $\Sigma$ and $\mathbb{D}$ be the essential set of $(M^3,g)$. It is obvious that we can choose $\rho' \leq \rho$ such that  $Vol(\mathcal{D}\setminus\mathbb{D}_\rho )+ Vol(\mathbb{D}_{\rho'})=Vol (\mathcal{D})$. Then, By the definition of isoperimetric surface, we have
$$
Area(\p (\mathcal{D}\setminus\mathbb{D}_\rho ))+Area(\p \mathbb{D}_{\rho'})\geq Area(\Sigma).
$$
Note that,
$$
\p (\mathcal{D}\setminus\mathbb{D}_\rho )\subset (\Sigma\setminus\mathbb{D}_\rho )\cup\p \mathbb{D}_\rho,\ \
Area(\p \mathbb{D}_{\rho'})\leq Area(\p \mathbb{D}_\rho).
$$

Furthermore, we have for any $\rho>0$
$$\begin{array}{rcl}
&&Area(\Sigma \cap \mathbb{D}) + Area((\Sigma \setminus \mathbb{D}))\\
&=&Area(\Sigma)\\
&\leq & Area(\partial (\mathcal{D}\setminus\mathbb{D}_\rho ))+Area(\p \mathbb{D}_{\rho'})\\
&\leq & Area((\Sigma\setminus\mathbb{D}_\rho )\cup\p \mathbb{D}_\rho)+Area(\p \mathbb{D}_{\rho})\\
&\leq & Area((\Sigma\setminus\mathbb{D}_\rho ))+ Area(\p \mathbb{D}_\rho)+Area(\p \mathbb{D}_{\rho})\\
&\leq & Area((\Sigma\setminus\mathbb{D}_\rho ))+ 2Area(\p \mathbb{D}_\rho).\\
\end{array}$$
Hence, combining the inequality above, we have

$$
Area(\Sigma \cap \mathbb{D}_\rho)\leq 2 Area(\p \mathbb{D}_\rho)\leq \Lambda e^{2\rho}.
$$
It implies the conclusion, we finish the lemma.

\end{proof}

\vskip 5mm

Let $\Sigma$ be a connected isoperimetric surface in $(M^3, g)$, $A$, $H$ denote its second fundamental form and mean curvature with respect to the outward unit normal vector respectively and ${\AA}\triangleq A-\frac{H}2 g_\Sigma$ be the trace free part of $A$. Here and in the sequel, $g_\Sigma$ denotes the induced metric on $\Sigma$ from $g$. Then, we have

\begin{lemm}\label{squareintergableofA0}
Let $\Sigma_i$ be a family of connected isoperimetric surfaces in  AH manifold $(M^3, g)$ with $H\geq 2$ and $Area (\Sigma_i)\rightarrow \infty$. Then, there exist universal constant $\Lambda_k$, $k=1,2$ depending only on $(M,g)$ such that

$$ \
\ \ g(\Sigma_i)\leq \Lambda_1,\ \ \ \  \ \ \ \int_{\Sigma_i}(\|{\AA}\|^2 +(H^2-4)) \ d\sigma_i \leq \Lambda_2.
$$
Here and in the sequel, $g(\Sigma_i)$ denotes the genus of $\Sigma_i$,
\end{lemm}
\begin{proof}
We will adopt Hersch's technique to prove the lemma, It was proved in \cite{NT} when the topology of $\Sigma_i$ is $\mathbb{S}^2$. Let
$$
\Psi_i=(\Psi_i^1,\Psi_i^2, \Psi_i^3 ): \Sigma_i \mapsto\mathbb{S}^2 \hookrightarrow \mathbb{R}^3
$$
be a conformal map with degree of $d_{\Sigma_i}$ and
$$
\int_{\Sigma_i}\Psi_i ^j=0, \qquad j=1,2,3.
$$
Noting that $\Sigma_i$ is volume preserving stable, we have

$$
\int_{\Sigma_i}(Ric(v)+\|A\|^2)|\Psi_i^j|^2 \leq \int_{\Sigma_i}\|\nabla \Psi_i^j\|^2,\ \ j=1,2,3.
$$
Here and in the sequel, $v$ is the outward unit normal vector of $\Sigma_i$. Hence

$$
\int_{\Sigma_i}(Ric(v)+\|A\|^2)\leq \int_{\Sigma_i}\|\nabla \Psi_i\|^2 =8\pi d_{\Sigma_i}.
$$
Due to Brill-Noether Theorem(see \cite{GH}), we can choose $\Psi_i$ with
$$
d_{\Sigma_i}\leq (1+[\frac{g(\Sigma_i)+1}2]).
$$
Therefore, we get

$$
\int_{\Sigma_i}(Ric(v)+\|A\|^2)\leq 8\pi \left(1+(\frac{g(\Sigma_i)+1}2)\right)=4\pi g(\Sigma_i)+12\pi.
$$

Let $e_1, e_2$ be the unit tangent vector of $\Sigma_i$ and $K$ be its Gauss curvature, then we have

\begin{equation}\label{estimate1}
\begin{split}
\int_{\Sigma_i}\left(Ric(v)+\|A\|^2\right)&=\int_{\Sigma_i}\left(Ric(e_1)+Ric(e_2)+H^2\right)-2\int_{\Sigma_i}K\\
&\leq 4\pi g(\Sigma_i)+12\pi,
\end{split}
\end{equation}
By Gauss-Bonnet formula, we obtain
\begin{equation}\label{integralofK}
\int_{\Sigma_i} K =4\pi\left(1- g(\Sigma_i)\right),
\end{equation}
On the other hand, we have
\begin{equation}\label{estimate2}
\begin{split}
\int_{\Sigma_i}(Ric(e_1)+Ric(e_2)+H^2)&=\int_{\Sigma_i}(H^2 -4)+\int_{\Sigma_i}O(e^{-3\rho})\\
&\geq \int_{\Sigma_i}O(e^{-3\rho})
\end{split}
\end{equation}
Here, we have used the assumption of $H\geq 2$ in the inequality above.
\\

Next, we will show that there exist a universal constant $\Lambda_3$ depending only on $(M,g)$ and $\mathbb{D}$ such that
\begin{equation}\label{finiteintegral1}
\int_{\Sigma_i}e^{-3\rho}\leq \Lambda_3.
\end{equation}
In fact, for any integer $k\geq 0$, let $\Sigma_{i,k}\triangleq \Sigma_i \cap(\mathbb{D}_{k}\setminus \mathbb{D}_{k-1}) $ and  $\mathbb{D}_{-1}=\emptyset$. Then
$$
\Sigma_i=\bigcup_{k}\Sigma_{i,k}.
$$

Due to Lemma \ref{areagrowthisosurf}, we see that there exists a constant $C$ independent of $k$. Then for each $k$, we have
$$
\int_{\Sigma_{i,k}}e^{-3\rho}\leq Ce^{-k}.
$$
Thus, we see that inequality (\ref{finiteintegral1}) is true.

Combining with (\ref{estimate1}), (\ref{integralofK}),
(\ref{estimate2}), (\ref{finiteintegral1}), we have
$$
g(\Sigma_i)\leq 5+C \int_{\Sigma_i}e^{-3\rho}\leq \Lambda.
$$
By a direct computation and noting that $H\geq 2$, we obtain

$$
Ric(v)+\|A\|^2 \geq (Ric(v)+2)+\|\AA\|^2=O(e^{-3\rho})+\|\AA\|^2.
$$
Then together with (\ref{estimate1}) and (\ref{estimate2}), we obtain

$$
\int_{\Sigma_i}\left(\|\AA\|^2 +(H^2-4) \right) d\sigma_i \leq \Lambda_2 .
$$
Thus, we finish the proof of Lemma \ref{squareintergableofA0}.

\end{proof}

In order to get more delicate estimate of isoperimetric surfaces, we need the following lemma proved in \cite{NT}:

\begin{lemm}\label{finiteintegoftangentialprho}
Let $\{\Sigma_i\}$ be a family of connected isoperimetric surfaces in AH manifold $(M^3, g)$ with the scalar curvature $R\geq -6$ and $v$ be its outward unit normal vector and $\rho$ be the distant function to the essential set $\mathbb{D}$. Then, we have

$$
\int_{\Sigma_i} \left(1-\langle v, \frac{\p}{\p \rho}\rangle\right)^2 \ d\sigma_i \leq C.
$$
Here, $C$ is a universal constant depending only on $(M^3,g)$.
\end{lemm}

\begin{proof}
Due to Proposition 3.4 in \cite{NT}, we have
\begin{equation}\label{laplacerhoonsigma}
\begin{split}
\Delta_\Sigma \rho &=(4-2\|(\frac{\p}{\p \rho})^T\|)e^{-2\rho} + (2-H)+(H-2)(1-\langle \frac{\p}{\p \rho} ,v \rangle)\\
&+(1-\langle \frac{\p}{\p \rho} ,v \rangle)^2+ O(e^{-3\rho}).
\end{split}
\end{equation}
where $(\frac{\p}{\p \rho})^T$ denotes the tangential projection of $\frac{\p}{\p \rho}$ on $T\Sigma$. Integrating (\ref{laplacerhoonsigma})
on $\Sigma$ and  together with Lemma \ref{squareintergableofA0} and formulae (\ref{finiteintegral1}), we get the conclusion.

\end{proof}

{\begin{lemm}\label{hgeq2}
Let $(M^3, g)$ be an AH manifold with  and $\Sigma$ be a connected isoperimetric surface in $(M^3, g)$. Then its mean curvature  $H> 2$ provided that $Area(\Sigma)$ is large enough.
\end{lemm}
\begin{proof}
As $\Sigma$ is a compact surface, Hence, there exists a $\mathbb{D}_r$ such that
$$\Sigma \subset \overline{\mathbb{D}}_r,\ p \in \Sigma\cap \partial \mathbb{D}_r.$$
By comparison theorem, we have

$$H(\Sigma)\geq H_{\mathbb{D}_r}(p).$$
By a direct calculation, we have

$$H_{\mathbb{D}_r}(p)= 2 + 2e^{-2r} -\frac{tr_{\sigma} h}{\sinh ^3 r}+O(e^{-4r}).$$
As $Area(\Sigma)$ is large enough, it is obvious that as $r$ is large enough, we obtain

$$H(\Sigma)\geq H_{\mathbb{D}_r}(p) > 2.$$

\end{proof}

Proposition \ref{curvatureestimateisoperisurf} asserts that a sequence of connected isoperimetric surfaces satisfying some natural assumptions
in an AH manifold $(M^3, g)$ with the scalar curvature $R\geq -6$ have uniformly bounded second fundamental forms. Moreover, if we assume  they pass through a fixed compact set $\mathbf{K}$ in $M$ and their areas approach to infinity, then, by the compactness theorem , we obtain that there exists subsequence of $\{\Sigma_i\}$ converging with multiplicity one to properly embedded, noncompact and complete surface in $(M^3, g)$ in $C^k$-local topology for any $k\geq 1$ (for details, see the arguments in the proof of Theorem 18 in \cite{Ros}). It is possible that the limiting surface may have more than one connected components. We denote $\mathbb{S}$ be any complete and noncompact connected components of the limit surface.  Our next goal is to investigate some basic facts of $\mathbb{S}$.\\

The following  proposition is on the curvature estimate of isoperimetric surfaces in  AH manifold $(M^3, g)$. Namely,

\begin{prop}\label{curvatureestimateisoperisurf}
Let $\Sigma$ be an isoperimetric surface in an AH manifold $(M^3, g)$ with mean curvature $0<H\leq \Lambda$. Then, there exists a constant $C$ depending only on $\Lambda$ and $(M^3, g)$ ( more specifically, $C^1-$ bound of curvature and lower bound of injective radius of $(M^3, g)$) such that
$$
\|A\|\leq C.
$$

\end{prop}

\begin{proof}
we prove this proposition by contradiction, we assume that the proposition is false. Then,  we can find a family of isoperimetric surface $\Sigma_i \subset M^3$  with $0< H|_{\Sigma_i}\leq \Lambda$,
 and $p_i \in\Sigma_i  $ such that $C^2_i =\|A\|(p_i)=\max \|A\|\rightarrow \infty$. Then we consider $\{(M^3, C^2_i g, p_i )\}$ converges to $(\mathbb{R}^3, g_{euc}, o)$ in $C^{2,\alpha}$- pointed Gromov-Hausdorff topology. Here, $g_{euc}$, $o$ are the standard Euclidean metric of $\mathbb{R}^3$ and a fixed point respectively. Note that $\Sigma_i$ is still an isoperimetric surface in $\{(M^3, C^2_i g, p_i )\}$ with constant mean curvature $C^{-1}_i H|_{\Sigma_i}\rightarrow 0$. By the same arguments in the proof of Proposition \ref{stableofs}, we get a complete and stable minimal surface $\Sigma_\infty$ in $\mathbb{R}^3$ with its second fundamental form $\|A\|(o)=1$ which contradicts with the well-known result of \cite{dP}(see also \cite{FS}). Hence, Proposition \ref{curvatureestimateisoperisurf} is true.
\end{proof}

\begin{lemm}\label{finittotalcurvature}
The limiting surface $\mathbb{S}$ is a properly embedded, noncompact and complete surface with constant mean curvature of $H=2$, $Area(\mathbb{D}_\rho \cap \mathbb{S})\leq C e^{2\rho}$ and
$$
\int_{\mathbb{S}}|K|<C.
$$
Here, $K$ denotes the Gauss curvature of $\mathbb{S}$ and $C$ is a universal constant depending only $(M,g)$.
\end{lemm}
\begin{proof}
Combining Lemma \ref{squareintergableofA0} and Lemma \ref{hgeq2}, we see that the limiting surface $\mathbb{S}$ is a  properly embedded and complete surface with mean curvature $H=2$ and
$$
\int_{\mathbb{S}} |\AA|^2<C.
$$
Combining with Gauss equation and (\ref{finiteintegral1}), we get
$$
\int_{\mathbb{S}}|K|<C.
$$
Furthermore, by Lemma \ref{areagrowthisosurf}, we see that
$$
Area(\mathbb{D}_\rho \cap \mathbb{S})\leq C e^{2\rho}.
$$
Therefore, we finish proving the lemma.
\end{proof}

Our next lemma is to assert that the limit surface $\mathbb{S}$ is stable in the following sense, the similar result in asymptotically flat case was obtained in \cite{MJ2013}.

\begin{prop}\label{stableofs}
For any $\phi \in C^\infty _0 (\mathbb{S})$, we have

$$
\int_{\mathbb{S}}(Ric (v)+\|A\|^2)\phi^2 \leq \int_{\mathbb{S}}|\nabla\phi|^2.
$$
\end{prop}

\begin{proof}
By a natural extension, we assume that $\phi$ is defined in the neighbourhood of $\mathbb{S}$. Then, by a restriction, we get $\phi_i \in C^\infty(\Sigma_i)$ with $\phi_i=\phi|_{\mathbb{S}}$ for $i$ large enough. And we notice that  $c_i=\frac{1}{Area(\Sigma_i)}\int_{\mathbb{S}}\phi_i\rightarrow 0$ as $i$ approaches to infinity and
$
\bar \phi_i =\phi_i-c_i
$
such that
 $$
 \int_{\Sigma_i}\bar \phi_i=0.
 $$
Then,  noticing that $\Sigma_i$  has least area among all surfaces enclosing the same volume as $\Sigma_i$ does, we have
$$
\int_{\Sigma_i}(Ric (v)+\|A\|^2)\bar \phi^2_i \leq \int_{\Sigma_i}|\nabla \bar \phi_i|^2.
$$
It implies

\begin{equation}\label{stableestimate1}
\begin{split}
&\int_{\Sigma_i}(Ric (v)+\|A\|^2)\phi_i ^2 +c^2 _i \int_{\Sigma_i}(Ric (v)+\|A\|^2)
-2c_i\int_{\Sigma_i}(Ric (v)+\|A\|^2)\phi_i \\
&\leq \int_{\Sigma_i}|\nabla \phi_i|^2.
\end{split}
\end{equation}
Due to Lemma \ref{squareintergableofA0} and (\ref{finiteintegral1}), we see that
$$
|\int_{\Sigma_i}(Ric (v)+\|A\|^2)|\leq \Lambda.
$$
Here, $\Lambda$ is a constant independent of $i$. Then, Taking $i$ tend to infinity, we get

$$
\int_{\mathbb{S}}(Ric (v)+\|A\|^2)\phi^2 \leq \int_{\mathbb{S}}|\nabla\phi|^2.
$$
Thus, we finish  proving the lemma.

\end{proof}

In particular, Proposition \ref{stableofs} has the following  interesting application. The similar result on the stability has been obtained in the asymptotically flat version in \cite{MJ2013}

\begin{coro}\label{strongstability}
The limiting surface $\mathbb{S}$ is strongly stable, i.e.
$$
\int_{\mathbb{S}}(Ric (v)+\|A\|^2)\phi^2 \leq \int_{\mathbb{S}}|\nabla\phi|^2.
$$
for any $\phi-C \in C^\infty_0 (\mathbb{S})$. Here, $C$ is any constant.
\end{coro}

\begin{proof}
We will use "logarithmic cut-off trick" (see P.121 in \cite{Sch83}) to prove this corollary. In fact, by Lemma \ref{finittotalcurvature} and Huber's theorem (see \cite{H}), we know that $\mathbb{S}$ is conformally equivalent to a  surface $\mathbb{S}$ obtained through deleting finite points from compact Riemann surface $\bar {\mathbb{S}}$. Without loss of generality, we can assume that  we take $\bar {\mathbb{S}}\setminus \{p\}=\mathbb{S}$. For simplicity, we assume $\phi-1 \in C^\infty _0 (\mathbb{S})$ (which is denoted by $C^\infty _0 (\bar {\mathbb{S}}\setminus \{p\})$ in the following). Let $B_{r_i}(p)$, $i=1$, $2$ be two geodesic balls with centered at $p$ and radius $r_i$ in $\bar {\mathbb{S}}$. Define

$$
\xi(r)=\left\{
         \begin{array}{ll}
           0& \text { $r\leq r_2$} \\
          \frac{\log r-\log r_2}{\log r_1 -\log r_2}& \text { $r \in [r_2, ~r_1]$} \\
          \phi& \text { $r\geq r_1$}
         \end{array}
       \right.
$$

Choosing a suitable Lipschtiz function $\xi$ with compact support set and together with Lemma \ref{stableofs}, we obtain
$$
\int_{\mathbb{S}}(Ric (v)+\|A\|^2)\xi^2 d\sigma \leq \int_{\mathbb{S}}|\nabla\xi|^2 d\sigma
$$
Here, $d\sigma$ is volume element with respect to metric $g|_\mathbb{S}$. Hence, we have
\begin{equation}
\begin{split}
&\int_{\bar {\mathbb{S}}\setminus B_{r_1}(p)}(Ric (v)+\|A\|^2)\xi^2 d\sigma +\int_{ B_{r_1}\setminus B_{r_2} (p)}(Ric (v)+\|A\|^2)\xi^2 d\sigma\\
&\leq \int_{\bar {\mathbb{S}}\setminus B_{r_1}(p)}|\nabla\xi|^2 d\sigma +\int_{ B_{r_1}\setminus B_{r_2} (p)}|\nabla\xi|^2 d\sigma\\
&=\int_{\bar {\mathbb{S}}\setminus B_{r_1}(p)}|\nabla\phi|^2 d \sigma +\int_{ B_{r_1}\setminus B_{r_2} (p)}|\nabla\xi|^2 d\bar{\sigma}.
\end{split}
\end{equation}
Here, $d\bar{\sigma}$ denotes the volume element in Riemannian surface $\bar {\mathbb{S}}$.  Note that $\xi$ is bounded, and
$$
\int_{ B_{r_1}\setminus B_{r_2} (p)}|\nabla\xi|^2 d\bar{\sigma} \leq \frac{C}{\log \frac{r_1}{r_2} }.
$$
Here, we have used the conformal invariance of Dirichlet integral. Take $r_1$, and $r_2$ sufficiently small and its ratio sufficiently large, we see
that the above integral approaches to zero. Together with Lemma \ref{squareintergableofA0}, we obtain

$$
\int_{\mathbb{S}}\left(Ric (v)+\|A\|^2\right)\phi^2 \leq \int_{\mathbb{S}}|\nabla\phi|^2.
$$
\end{proof}

\section{No drift off to the infinity}

In this section, we will show that a connected isoperimetric region with uniformly positive lower bound of volume in AH manifold $(M^3, g)$ with scalar curvature $R\geq -6$ cannot drift off to the infinity provided that $(M^3, g)$ is not isometric to $\mathbb{H}^3$. More precisely, we have

\begin{prop}\label{nodrift}
Let$(M^3, g)$ be an AH with $R(g)\geq -6$ which is not isometric to $\mathbb{H}^3$. Let $\{D_i\}$ be a family of connected isoperimetric regions with $\mathcal{H}_g ^3(D_i)\geq \delta_0$. Here, $\mathcal{H}_g ^3(\ \cdot \  )$ is the Lebesgue measure on $(M^3, g)$ with respect to metric $g$ and $\delta_0$ is a positive fixed constant. Then, $\{D_i\}$ cannot drift off to the infinity of $(M^3, g)$ i.e. There is a fixed compact domain $E$ so that each $D_i$ intersects $E$.
\end{prop}

In order to prove Proposition \ref{nodrift}, we need to introduce the following notions.

\begin{defi}\label{isoprofile}
The isoperimetric profile of $(M^3,g)$ with volume $v$ is defined as
\begin{eqnarray*}
A(v)&=&\inf\{\mathcal{H}^2(\partial^{\ast}\Omega):\Omega\subset M ~is~ a~ Borel~ set~ with~\\
 &&finite ~perimeter~ and~ \mathcal{H}_g ^3(\Omega)=v\}.
\end{eqnarray*}
 Here, $\mathcal{H}^2(\ \cdot\ )$ is a 2-dim Hausdorff measure for the reduced boundary of $\Omega$ denoted by $ \partial^{\ast}\Omega$. A Borel set $\Omega\subset M$ of finite perimeter such that $\mathcal{H}_g ^3(\Omega)=v$
and $A(v)=\mathcal{H}^2(\partial^{\ast}\Omega)$ is called an isoperimetric region of $(M,g)$ of volume $v$. The surface $\partial\Omega$ is called isoperimetric surface.
\end{defi}

The main argument of proof of Proposition \ref{nodrift} comes from \cite{BC} and \cite{OC}. The following proposition is crucial to us and also has its own interest.

\begin{prop}\label{areacomparison}
 Suppose that $(M^3,g)$ is an AH manifold with $R\geq -6$. Then, $A_g(v)\leq A_{{\mathbb{H}}}(v)$ for every $v>0$. Moreover, if there exists a $v_0>0$ satisfying $A_g(v_0)= A_{{\mathbb{H}}}(v_0)$, then $(M^3,g)$ is isometric to $\mathbb{H}^3$.
\end{prop}

\begin{rema} In fact, similar result is still true as $v$ approaches to infinity(For details, see Theorem 1 in \cite{BC});
  For the case of asymptotically flat, we still have similar result(see Theorem 1.2 and Theorem 1.3 in \cite{Shi}).
\end{rema}

As in \cite{BC} and \cite{OC}, we will make use of inverse mean curvature flow to investigate Proposition \ref{areacomparison} (see also \cite{Shi} for asymptotically flat manifolds case). In facts, the idea and argument are from \cite{BC}. However, for the convenience of application in our context, we proceed as the second author did in \cite{Shi}. A classical solution of the inverse mean curvature flow is a smooth family $F: N\times [0,T]\rightarrow M$ of embedded hypersurfaces $N_t = F(N, t)$ satisfying the following evolution equation
\begin{eqnarray}\label{imcf}
\frac{\partial F}{\partial t}=\frac{\nu}{H},~~~~~~~~~~0\leq t\leq T,
\end{eqnarray}
where $H$ is the mean curvature of $N_t$ at $F(x, t$) with respect to the outward
unit vector $\nu$ for any $x\in N$. Specifically, Hawking mass plays an important role in the theory of inverse mean curvature flow.

\begin{defi}\label{hawkingmass1}
The Hawking mass is of a surface $\Sigma$ is defined as

\begin{equation}\label{mass}
m_H(\Sigma)=\frac{\mathcal{H}^2(\Sigma)^{\frac{1}{2}}}{(16\pi)^{\frac{3}{2}}}\left(16\pi-\int_{\Sigma}(H^2-4)d\mu\right).
\end{equation}
\end{defi}

Generally, the evolution equation \eqref{imcf} has no
classical solution. In order to overcome this difficulty, Huisken and Ilmanen introduced a level-set formulation of \eqref{imcf} in the setting of asymptotically flat manifolds (\cite{HI})where the evolving surfaces are given as level-sets of a scalar function $u$ via $N_t = \partial\{x : u(x) < t\}$ and $u$ satisfies the following elliptic equation in weak sense
\begin{eqnarray}\label{weaksolution}
div(\frac{\nabla u}{|\nabla u|})=|\nabla u|.
\end{eqnarray}

We note that the similar argument works well in AH case. More precisely, by Theorem 4.1 in \cite{OC}, let $B_\mu(x)$ be geodesic ball with any small radius $\mu>0$ and center $x$  in $(M, g)$ and $\Sigma=\p B_\mu(x)$. Then there exists weak solution of inverse mean curvature flow $u$ with initial condition $\{u=0\}=\Sigma$ and satisfying all other properties listed in Theorem 4.1 in \cite{OC}, as proof in Lemma 8.1 in \cite{HI}, we get $(G_t)_{-\infty <t <\infty}$ which is the weak solution of (\ref{imcf}) with single point $\{x\}$  as it's initial condition.

\begin{lemm}\label{imcfvol}
For any $v>0$ either there exists $t$ such that $Vol (G_t)=v$  or  $v$ is a jump  volume for (\ref{imcf}), i.e. there exits $t_1>-\infty$ such that
$$
Vol(G_{t_1})<v\leq Vol(G^+_{t_1}).
$$
Here, $G^+_{t_1}$ is the strictly minimizing hull for $G_{t_1}$.
\end{lemm}
and

\begin{lemm}\label{functiontofv}
For any $v>0$, let
\begin{equation}
t(v)=\inf\{\tau: Vol(G_\tau)\geq v\}\nonumber.
\end{equation}
Then, $t(v)$ is a Lipschitz function and
\begin{equation}
\frac{dt}{dv}\leq \left(\int_{\Sigma_t}H^2 \right )^\frac12 \cdot \left( Area(\Sigma_t) \right)^{-\frac32}.\nonumber
\end{equation}
Here, $\Sigma_t=\p G_t$.
\end{lemm}

For the proof of two lemmas above (see Lemma 3.3 and Lemma 3.4 in \cite{Shi} respectively). For each $t$, we let $B(t)=Area(\Sigma_t)$. Then, by Lemma \ref{functiontofv}, we regard $B$ denoted by $B(v)$ as the function of $v$ and $m(v)$ as the Hawking mass of $\Sigma_{t(v)}$.
Then, by the same arguments in the proof of Theorem 1.2 and Theorem 1.3 in \cite{Shi}, we see that

\begin{prop}\label{derivativebtov}
For each $v\geq 0$, we have
\begin{equation}\label{BV}
\frac{d}{dv}B(v)\leq B^{-\frac{1}{2}}(v)\left(16\pi+4B(v)-\frac{(16\pi)^{\frac{3}{2}}}{B^{\frac{1}{2}}(v)}m(v)\right)^{\frac{1}{2}}.
\end{equation}
\end{prop}
As in \cite{Shi}, we consider the region $M_{e}\triangleq M \setminus \mathbb{D}$ of $M$. By the definition of AH manifold, without loss of generality, we assume that $M_e$ is diffeomorphic to $\mathbb{R}^3 \setminus B_1(o)$. Let $\Omega_{e}=\Omega \cap M_{e}$ and
\begin{eqnarray*}
A_e(v)&=&\inf\{\mathcal{H}^2(\partial^{\ast}\Omega_e):\Omega\subset M ~is~ a~ Borel~ set~ with~\\
 &&finite ~perimeter~ and~ \mathcal{H}^3(\Omega_e)=v\}.
\end{eqnarray*}
Clearly, we have $A(v)\leq A_e(v)$. In the following, we mainly focus on $A_e(v)$.

\begin{lemm}\label{nondecreasingav}
Let $(M^3,g)$ be an AH manifold, then $A_e(v)$ is nondecreasing.
\end{lemm}

\begin{rema}Similar result was proved in \cite{OC}, see Lemma 3.3 therein, but in current case the mean curvature of the $\p M_e$ may not  equal to $2$, so we have to handle this carefully.
\end{rema}\par

\begin{proof}[\textbf{Proof of Lemma \ref{nondecreasingav}}]
For $\mathbb{D}$ large enough. By the definition of AH manifold and direct computation,  we have
\begin{equation}\label{finitevoldifference}
\int_{M_e}|\sqrt{g}-\sqrt{g_\mathbb{H}}|dv_g \leq 1.
\end{equation}
Here, $dv_g$ denotes the volume element with respect to metric $g$. We show the lemma by contradiction.

We assume that $A_e(v)$ is not nondecreasing, it means that there exists $$v_1 < v_2,$$ but$$A_e(v_1)> A_e(v_2).$$
Furthermore, we define
$$
\mu =\ {\mbox{inf}}\ \{\mathcal{H}^2 (\p^* \Omega_e): \mathcal{H}_g ^3 (\Omega_e)\geq v_1 \}.
$$

It is obvious that $\mu \leq  A(v_2)$. We claim that $\mu$ can be achieved by certain $\Omega'_e \subset M_e$ and a hyperbolic ball in $B(S)\subset\mathbb{H}^3$ (see Proposition 3.2 in \cite{OC}). i.e. there exists a $v\geq v_1$,  $\Omega'_e \subset M_e$ and $S\geq 0$ such that
\begin{itemize}
  \item $ \mathcal{H}_g ^3 (\Omega'_e)+\mathcal{H}_\mathbb{H} ^3 (B(S))=v$,
   \item $\mathcal{H}_g ^2 (\p^* \Omega' _e )+S=\mu.$
\end{itemize}
  Here, $B(S)$ is the geodesic ball in $\mathbb{H}^3$ with its area equals to $S$, $\mathcal{H}_\mathbb{H} ^3(\ \cdot\ )$ is the Lebesgue measure in $\mathbb{H}^3$. Then we choose $\{\Omega_e ^i\}$ as corresponding minimizing sequence. Note that $\mathcal{H}_g ^3 (\Omega_e ^i)$ is uniform bound. Otherwise, Through direct computation and (\ref{finitevoldifference}), we have
$$
\mathcal{H}_\mathbb{H} ^3(\Omega_e ^i )\geq \mathcal{H}_g ^3 (\Omega_e ^i)-1.
$$
Hence. we obtain that $\mathcal{H}_\mathbb{H} ^3 (\Omega_e ^i)$ would be unbounded and its area with respect to $g_\mathbb{H}$ is also unbounded, and then implies $\mathcal{H}^2 (\p \Omega_e ^i)$ would be unbounded which reach the contradiction. Once we know $\mathcal{H}_g ^3 (\Omega_e ^i)$ is uniformly bounded, then by the arguments in \cite{Br} we see claim is true.

Next, we claim that  $v>v_1$.  Indeed if our claim is false, then we have $v=v_1$. We firstly prove that $S=0$. If not, we have $S>0$. Then we can put $B(S)$ in $M_e$  which is far away from $\mathbb{D}$ where $g$ is very close to $g_\mathbb{H}$. Hence, $\mathcal{H}_g ^3 (B(S))$ and $\mathcal{H}^2 (\p B(S) )$ can be very close to $\mathcal{H}_\mathbb{H}^3 (B(S))$  and $S$ respectively, then for  $\epsilon>0$ small enough and by a small perturbation on $\Omega' _e \cup B(S)$ in $M_e$ if necessary, we can construct a domain $\bar \Omega \subset M_e$ with $$\mathcal{H}_g ^3 (\bar \Omega)=v=v_1$$ and
$$\mathcal{H}^2 (\p^* \bar \Omega )\leq \mu+\epsilon \leq A_e(v_2)+\epsilon < A_e(v_1).$$
It contradicts with the definition of isoperimetric profile $A_e(v)$. Hence, we have $$S=0.$$

However, as $S=0$, we have $$\mathcal{H}_g^2(\p^* \Omega^\prime_{e}) =\mu \leq A_e(v_2)<A_e(v_1).$$
It still contradicts with the definition of isoperimetric profile $A_e(v)$. Hence, we finally have
 $$v >v_1.$$

 Then $\p\Omega'_e \setminus \mathbb{D}$ is a stable minimal surface if it's nonempety. However, this is impossible as the level set of distant function to $\mathbb{D}$ has positive mean curvature. Thus, we conclude that $\mathcal{H}_g^3(\Omega_e^\prime)=0$ and $\mathcal{H}_g^2(\p^* \Omega_e^\prime)\geq 0$. Then we have

 \begin{itemize}
  \item $\mathcal{H}_\mathbb{H} ^3 (B(S))=v>v_1,$
   \item $\mathcal{H}_\mathbb{H} ^2 (\p^* B(S))=S\leq \mu \leq A_e(v_2) < A_e(v_1).$
\end{itemize}
Henceforth, there exists $S^\prime <S $ such that
$$\mathcal{H}_\mathbb{H} ^3 (B(S^\prime))=v_1,\ \  \mathcal{H}_\mathbb{H} ^2 (\p^* B(S^\prime))=S^\prime<S\leq \mu \leq A_e(v_2) < A_e(v_1).$$
Then, By making a small perturbation, we can construct a region $D \subset M_e$ far away from $\mathbb{D}$ (this trick being used above)such that
$$\mathcal{H}_g ^3 (D)=v_1,\ \  \mathcal{H}_g ^2 (\p^* D) < A_e(v_1).$$
It contradicts with the definition of isoperimetric profile of volume $v_1$. Thus, we finish the proof of  Lemma \ref{nondecreasingav}.

\end{proof}
\vskip 5mm
Now, we are in the position to prove Proposition \ref{areacomparison}:

\begin{proof}[\textbf{Proof of Proposition \ref{areacomparison}}] We firstly set
$$f(v)=B^\frac32 (v),\ \ f_{\mathbb{H}}(v)=A_{\mathbb{H}}^{\frac32} (v)$$
where $A_{\mathbb{H}}( \cdot )$ denotes the isoperimetric profile of $\mathbb{H}^3$. We want to prove that
$$\forall \epsilon>0,\  \forall v\geq 0,\ \
f(v)\leq (1+\epsilon)f_{\mathbb{H}}(v).$$

 we prove it by contradiction, we suppose that the results is false, then there exists $v_1>0$ such that $$f(v_1)>(1+\epsilon)f_{\mathbb{H}}(v_1).$$

 On the other hand, we know that there exists some $\delta>0$ (may depend on $\epsilon$) such that $f(v)\leq (1+\epsilon)f_{\mathbb{H}}(v)$, $ \forall v\leq \delta$. Then, we set
$$
v_0=\inf\{v:f(v)>(1+\epsilon)f_{\mathbb{H}}(v) \},
$$
Then, we have $v_0\geq\delta >0$ and $f(v_0)=(1+\epsilon)f_{\mathbb{H}}(v_0)$. Setting $\omega(v)=f(v)-(1+\epsilon)f_{\mathbb{H}}(v)$,
we have
\begin{equation}\label{omega}
\begin{split}
\frac{d}{dv}\omega(v)\leq 3(4\pi +f^\frac23(v))^\frac12 -3(4\pi+(\frac1{1+\epsilon})^\frac23 ((1+\epsilon)f_{\mathbb{H}}(v))^\frac23)^\frac12 (1+\epsilon)
\end{split}
\end{equation}
Here, we have used Proposition \ref{derivativebtov}, $m(v)\geq 0$ and
$
\frac{d}{dv}f_{\mathbb{H}}(v)= 3\left(4\pi +f_{\mathbb{H}}^\frac23(v)\right)^\frac12,
$

In the formula  (\ref{omega}). Noticing that $\omega$ is Lipschtiz, we can choose a sequence of $\{\alpha_i >0\}$ and $\lim\limits_{i\rightarrow \infty}\alpha_i=0$. Then, As $i$ becomes large enough, we have

\begin{equation}
\begin{split}
0&\leq \frac{\omega(v_0)-\omega(v_0 -\alpha_i)}{\alpha_i}\\
&\leq \frac3{\alpha_i}\int^{v_0}_{v_0 -\alpha_i}\left((4\pi +f^\frac23(v))^\frac12 -[4\pi+(\frac1{1+\epsilon})^\frac23 ((1+\epsilon)f_{\mathbb{H}}(v))^\frac23]^\frac12 (1+\epsilon)\right)dv\\
&<0.
\end{split}
\end{equation}
we reach a contradiction. Hence, we have that $f(v)\leq (1+\epsilon)f_{\mathbb{H}}(v)$ as $\epsilon$ is arbitrary. It means that  $f(v)\leq f_{\mathbb{H}}(v)$ or $B(v)\leq A_{\mathbb{H}}(v)$ for $\forall v\leq \delta$.

Now we begin to prove  $A(v)\leq A_{\mathbb{H}}(v)$ as follows. For $\forall v>0$, we  choose sufficiently large $\rho_0=\rho_0(v)>0$ and for any $x\in M_e$, we consider the inverse mean curvature flow with initial data $\{x\}$. By choosing $\rho_0$ sufficiently large if necessary, we  assume $G_t\subset M_e$ with $\mathcal{H}^3(\Omega_t)>v$, here $G_t$ is the compact region bounded by $\Sigma_t$ and $\Sigma_t$ is the weak solution of the inverse mean curvature flow with $\{x\}$ as the initial condition. Due to discussion above, It's obvious that  $B(v)\leq A_{\mathbb{H}}(v)$. If $v$ is not a jump volume, then there exists $t$ such that $G_t$ with $Vol(G_t)=v$. Hence, we have
$$
A(v)\leq A_{e}(v)\leq Area(\Sigma_t)=B(v)\leq A_{\mathbb{H}}(v).
$$
Otherwise, $v$ is a jump volume. At this time, there exists $G_\tau$ such that $$v_1 =Vol (G_\tau)<v \leq Vol (G^+ _\tau)=v_2.$$ Hence, $t(v)=\tau$ and $B(v)=B(v_1)$,

$$
A(v)\leq A_{e}(v)\leq A_{e}(v_2)\leq Area (\Sigma^+ _\tau)=Area(\Sigma_\tau)=B(v_1)=B(v).
$$
Here, we have used Lemma \ref{nondecreasingav} in the second inequality above and $\Sigma^+ _\tau =\p G^+ _\tau$.
\\

Next, we claim that {\it if for some $v_0>0$, $A(v_0)= A_{\mathbb{H}}(v_0)$, then $(M^3,g)$ is isometric to $\mathbb{H}^3$}.
 Suppose not, then there exists $x\in M$, $Ric(x)\neq-2g$. Considering the weak solution of inverse mean curvature flow with initial condition $\{x\}$,  we have $m(v)>0$ for $\forall v>0$. Hence, $A(v_0)\leq A_e(v_0)\leq B(v_0)< A_{\mathbb{H}}(v_0)$ which is a contradiction. Therefore, we have proved that $(M^3,g)$ is isometric to $\mathbb{H}^3.$
\end{proof}

As in \cite{OC}, the following lemma is used to prevent an isoperimetric region with large volume from drifting off to the infinity of $(M,g)$.

\begin{lemm}\label{largev} Let $(M^3,g)$ be an AH with $R(g)\geq -6$, we have
\begin{eqnarray*}
A(v)\leq  A_{\mathbb{H}}(v)-2V(M,g)+2\sqrt{2}\pi^{\frac{3}{2}}\left(\int_{\mathbb{S}^2}(tr_{\sigma}h)\right)v^{-\frac{1}{2}}+o(v^{-\frac{1}{2}}), \  v\rightarrow \infty.
\end{eqnarray*}
Here, $V(M,g)$ is the renormalized volume.
\end{lemm}
\begin{proof}

we firstly choose sufficiently large $\rho_0$. Let $v_1=\mathcal{H}^3_g(\Gamma_{\rho_0})$ and $v_2=\mathcal{H}_{\mathbb{H}}^3(\Gamma_{\rho_0})$, here $\Gamma_{\rho_0}=\{\rho\leq\rho_0\}$.
According to the definition of asymptotically hyperbolic manifold, we know on $M\backslash K$, here $K\subset M$ is some compact set
$$\|g-g_{\mathbb{H}}\|_{g_{\mathbb{H}}}=O(e^{-3\rho}).$$
It is not hard to check that
\begin{eqnarray*}
\mathcal{H}^2_{\mathbb{H}}(\partial \Gamma_{\rho} ) &=&4\pi \sinh^{2}\rho,\\
\mathcal{H}^2_g(\partial \Gamma_{\rho} ) &=&4\pi \sinh^{2}\rho(1+\frac{1}{6\sinh^3\rho}\left(\int_{\mathbb{S}^2}(tr_{\sigma}h)\right)+o(e^{-\rho}).
\end{eqnarray*}
By the definition of renormalized volume, we have
\begin{eqnarray*}
\lim_{\rho\rightarrow\infty}(\mathcal{H}^3_g(\Gamma_{\rho})-\mathcal{H}_{\mathbb{H}}^3(\Gamma_{\rho}))=V(M,g),
\end{eqnarray*}
i.e.
\begin{eqnarray*}
\int_{\rho_0}^{\infty}[\mathcal{H}^2_g(\partial \Gamma_{\rho} )-\mathcal{H}^2_{\mathbb{H}}(\partial \Gamma_{\rho} )]d\rho+v_1-v_2=V(M,g).
\end{eqnarray*}
Thus
\begin{eqnarray}\label{v}
&&\mathcal{H}^3_g(\Gamma_{\rho})-\mathcal{H}_{\mathbb{H}}^3(\Gamma_{\rho})-V(M,g)\nonumber\\
&=&-\int_{\rho}^{\infty}[\mathcal{H}^2_g(\partial \Gamma_{\rho} )-\mathcal{H}^2_{\mathbb{H}}(\partial \Gamma_{\rho} )]d\rho\nonumber\\
&=&-\int_{\rho}^{\infty}\frac{2\pi}{3\sinh\rho}d\rho \cdot(\int_{\mathbb{S}^2}(tr_{\sigma}h))+o(e^{-\rho})\nonumber\\
&=&-\frac{2\pi}{3\sinh\rho}\cdot(\int_{\mathbb{S}^2}(tr_{\sigma}h))+o(e^{-\rho}).
\end{eqnarray}
For simplicity, we also denote $\Gamma_{v}=\{\rho\leq \rho_v\}$ where $\rho_v$ is chosen so that $\mathcal{H}^3_g(\Gamma_v)=v$.
It is well-known that the coordinate sphere in hyperbolic space is the unique isoperimetric surface. Choose $\rho'_v$ such that
\begin{eqnarray*}
A_{{\mathbb{H}}}(v)=4\pi \sinh^{2}\rho'_v,~~~~~~~~~~and~~~~~~~~~~v=\int_{0}^{\rho'_v}4\pi \sinh^{2}\rho d\rho.
\end{eqnarray*}
On the other hand, by \eqref{v}, we have
\begin{eqnarray*}
(\mathcal{H}^3_g(\Gamma_v)-\mathcal{H}_{\mathbb{H}}^3(\Gamma_v))-V(M,g)=-\frac{2\pi}{3\sinh\rho_v}(\int_{\mathbb{S}^2}(tr_{\sigma}h))+o(e^{-\rho_v}),
\end{eqnarray*}
i.e.
\begin{eqnarray*}
\left(v-\int_{0}^{\rho_v}4\pi \sinh^{2}\rho d\rho\right)-V(M,g)=-\frac{2\pi}{3\sinh\rho_v}(\int_{\mathbb{S}^2}(tr_{\sigma}h))+o(e^{-\rho_v}).
\end{eqnarray*}
By direct computation,
\begin{eqnarray*}
\int_{0}^{\rho'} \sinh^{2}\rho d\rho=\frac{1}{2}\sinh^{2}\rho'+\frac{1}{4}-\frac{\rho'}{2}-\frac{1}{4}e^{-2\rho'}.
\end{eqnarray*}
Thus
\begin{eqnarray*}\begin{array}{lll}
&4\pi \sinh^{2}\rho'_v-4\pi \sinh^{2}\rho_v&\\
=&2V(M,g)-\frac{4\pi}{3\sinh\rho_v}\left(\int_{\mathbb{S}^2}(tr_{\sigma}h)\right)+o(e^{-\rho'_v})+o(e^{-\rho_v})+4\pi (\rho'_v-\rho_v),
\end{array}
\end{eqnarray*}
we let $v\rightarrow\infty$, so we have
$$\rho'_v\rightarrow\infty,\ \rho_v\rightarrow\infty,\ \pi (e^{2\rho'}_v-e^{2\rho_v})-4\pi(\rho'_v-\rho_v)\rightarrow 2V(M,g).
$$
It means that
\begin{eqnarray*}
\pi e^{2\rho_v} (e^{2(\rho'_v-\rho_v)}-1)-4\pi(\rho'_v-\rho_v)\rightarrow 2V(M,g).
\end{eqnarray*}
so
$$e^{2(\rho'_v-\rho_v)}-1=O(e^{-2\rho_v}),$$
and
$$\rho'_v-\rho_v=O(e^{-2\rho_v}).$$
Thus
\begin{eqnarray*}
A_g(v)-  A_{{\mathbb{H}}}(v)&\leq&\mathcal{H}^2_g(\partial \Gamma_v )-4\pi \sinh^{2}\rho'_v\\
&\leq&4\pi \sinh^{2}\rho_v\left(1+\frac{\pi (\int_{\mathbb{S}^2}(tr_{\sigma}h))}{6\sinh^3\rho_v}\right)-4\pi \sinh^{2}\rho'_v+o(e^{-\rho_v})\\
&\leq&
-2V(M,g)+2\sqrt{2}\pi^{\frac{3}{2}}\left(\int_{\mathbb{S}^2}(tr_{\sigma}h)\right)v^{-\frac{1}{2}}+o(v^{-\frac{1}{2}}),
\end{eqnarray*}
Thus, we conclude to prove the Lemma.

\end{proof}

Now we can prove Proposition \ref{nodrift}.

\begin{proof}[\textbf{Proof of Proposition \ref{nodrift}}]
Suppose that $\{D_i\}$ is a family of connected isoperimetric regions in an AH manifold $(M^3, g)$ with $R\geq -6$ and $\mathcal{H}_g ^3 (D_i)\geq \delta_0 >0, \delta_0 \in \mathbb{R}$, we want to show that for a fixed point $o\in M $, there exists a constant $\Lambda>0$ such that $d(o, D_i)\leq \Lambda$, for all $i$. Otherwise, we can find a subsequence, which is still denoted by $\{D_i\}$ with $d(o, D_i)\rightarrow \infty$. Noticing  that $(M^3,g)$ is not isometric to $\mathbb{H}^3$, we have $V(M,g)>0$.  We can use the exact same arguments to prove this fact. The only difference between our current case and paper \cite{BC} is that a single point rather than a horizon in $M$ is taken as the initial data for the inverse mean curvature flow which  played an important role in \cite{BC}. Indeed,  our case can be regarded as the limiting case of $m =0$ therein. We consider the following two cases.

{\bf Case 1}: $v_i =\mathcal{H}_g ^3 (D_i)\rightarrow\infty$. Then due to Lemma \ref{largev}, we have
\begin{equation}\label{positivediffofvol}
0<V(M,g)\leq A_{\mathbb{H}}(v_i)-A(v_i),\ \  i \rightarrow \infty,
\end{equation}
On the other hand, by the definition of AH metric, we have
$$
\|g-g_{\mathbb{H}}\|_g \leq C e^{-3\rho}.
$$
Hence, by the similar arguments in the proof of (\ref{finiteintegral1}),
$$
A_{\mathbb{H}}(v_i)-A(v_i)\rightarrow 0, \quad \text{} \quad i\rightarrow\infty.
$$
which contradicts with inequality(\ref{positivediffofvol}).

{\bf Case 2}: $\delta_0\leq v_i= \mathcal{H}_g ^3 (D_i)\leq C <\infty$. Again, by the similar arguments in the proof of (\ref{finiteintegral1}), we have

$$
A_{\mathbb{H}}(v_i)-A(v_i)\rightarrow 0,  \quad i\rightarrow\infty,
$$
 By taking a subsequence if necessary, we assume $v_i\rightarrow v_0>0$,  we have $A_{\mathbb{H}}(v_0)=A(v_0)$. Then, we see that $(M^3, g)=\mathbb{H}^3$  by Proposition \ref{areacomparison}. It contradicts with our assumption.
 Hence, we finish the proof of  Proposition \ref{nodrift}.

\end{proof}

As a corollary of  Proposition \ref{nodrift}, we have
\begin{coro}
Let $(M^3, g)$ be an AH manifold with $R(g)\geq -6$ and be not isometric to $\mathbb{H}^3$. If $\{D_i\}$ is a family of connected isoperimetric regions with $\delta_0 \leq \mathcal{H}_g ^3(D_i)\leq \delta_1$ where $\delta_0, \delta_1$ are fixed positive constant and connected boundary. Then, there exists a compact and smoothly embedded surface $\Sigma$ with constant mean curvature such that $\p D_i$  converges to $\Sigma$ in topology of $C^k$, for any $k\geq 1$.
\end{coro}

\section{exhaustion of isoperimetric regions}
In the section, we are aimed at exploring some properties of isoperimetic regions in $(M^3,g)$. we always assume that $D_i$ are connected isoperimetric region with $\mathcal{H}^3_g(D_i)\geq \delta_0>0$ and $\Sigma_i=\partial D_i$ are topologically spherical isoperimetric surface in $M$.\par
Obviously, there are three cases for a family of isoperimetric regions in $(M,g)$ i.e.
\begin{enumerate}
  \item $\{D_i\}$ drift off to the infinity of $(M,g)$;
  \item $\{D_i\}$ are an exhaustion of $(M,g)$;
  \item $\{D_i\}$ always pass through some fixed compact domain.
\end{enumerate}

In the Proposition \ref{nodrift}, we proved that the case (1) cannot occur if  $M$ is not isometric to $\mathbb{H}^3$. Hence, we just deal with the case (2) and the case (3) in our rest part of this section.

\begin{theo}\label{rigidity}
Let $\mathbb{S}$ be the limiting surface of a family of isopermetric surfaces $\{\Sigma_i\}$ in an AH manifold $(M^3,g)$ with $R(g)\geq -6$ and $h=m\sigma$ in Definition \ref{definitionofahmfd} , if $\mathbb{S}$ is an noncompactly, completely connected surface with $\int_{\mathbb{S}}K\leq0$, then $(M^3,g)$ is isometric to $\mathbb{H}^3$.
\end{theo}

\begin{proof}
By  taking $\varphi=1$ in Corollary \ref{strongstability}, we obtain

$$
0\leq\int_{\mathbb{S}}(R+6+\|{\AA }\|^2) d\sigma_g\leq\int_\mathbb{S}K d\sigma_g \leq0,
$$
so
\begin{eqnarray}\label{zeroflat}
R=-6,~~~~~\AA =0~~~~~~~~~~~{\mbox{on}}~\mathbb{S},\ \int_{\mathbb{S}}K d\sigma _g = 0.
\end{eqnarray}
In the following, we prove $K=0$ by the same argument in \cite{FS}. For the convenience of reader, we sketch the main argument in \cite{FS}.
By (\ref{zeroflat}), stable operator
is reduced to
\begin{equation*}
 L=\Delta_{\mathbb{S}} -K.
\end{equation*}
Therefore, stable condition of $L$ implies that there exists a positive solution $f$ to the equation
$\Delta_{\mathbb{S}} f-Kf=0 $ on $\mathbb{S}$. Setting $w=\log f$, we have
$$\Delta_{\mathbb{S}} w=K-|\nabla w|^2.$$
For any function $\phi \in C^\infty_0(\mathbb{S})$, we have
 $$\int_{\mathbb{S}}|\nabla w|^2\phi^2d\sigma_g=\int_{\mathbb{S}}K\phi^2d\sigma_g+2\int_{\mathbb{S}}\langle\nabla \phi,\nabla w\rangle\phi d\sigma_g.$$
 By Schwarz inequality, we get
 $$\frac{3}{4}\int_{\mathbb{S}}|\nabla w|^2\phi^2d\sigma_g\leq \int_{\mathbb{S}}K\phi^2d\sigma_g+4\int_{\mathbb{S}}|\nabla \phi|^2d\sigma_g.$$
Note that $\mathbb{S}$ is conformally equivalent to a compact Riemannian surface deleted finite points, by using "logarithmic cut-off trick" as in proof of Corollary \ref{strongstability}, and by choosing suitable  $\phi \in C^\infty_0(\mathbb{S})$, we get
$$
\frac{3}{4}\int_{\mathbb{S}}|\nabla w|^2 \leq \int_{\mathbb{S}}K.
$$
It implies $w= constant $. Hence, $K=0$.\\

By Gauss equation
\begin{eqnarray*}
0=K&=&\frac{R}{2}-R_{\nu\nu}+\frac{H^2}{4}-\frac{ |{\AA }|^2}{2}\\
&=&\frac{m}{\sinh^3\rho}-\frac{3m\cdot|\partial_\rho^\top|^2}{2\sinh^3 \rho}+O(\exp(-4\rho))\\
&=&\frac{m}{\sinh^3\rho}(1-\frac{3|\partial_\rho^\top|^2}{2})+O(\exp(-4\rho)).
\end{eqnarray*}
Here, $\rho$ is the distant function to the essential set $\mathbb{D}$. If $m\neq 0$, we have for $\rho$ large enough

\begin{equation}\label{zerolimits}
1-\frac{3|\partial_\rho^\top|^2}{2}=O(e^{-\rho}),\  {\mbox{on}}\  \mathbb{S}.
\end{equation}
Then, by Proposition \ref{finiteintegoftangentialprho}, we have
$$\int_{\mathbb{S}}(1-\langle \nu,\frac{\partial \ }{\partial \rho}\rangle )^2 <\infty.$$
Hence, there exists $p_i\in \mathbb{S}, p_i\rightarrow \infty$ such that  $(1-\langle \nu,\frac{\partial \ }{\partial \rho}\rangle)(p_i)\rightarrow 0$. It means that $|\partial_\rho^\top|^2(p_i)\rightarrow 0$ which contradicts with (\ref{zerolimits}). Hence, $m=0$. By positive mass theorem in \cite{WXD}, $(M^3,g)$ is isometric to $\mathbb{H}^3$.
\end{proof}

Next, we begin to prove the following main result.

\begin{theo}\label{classification}
Let $(M^3,g)$ be an AH manifold with the scalar curvature $R(g)\geq -6$ and $h=m\sigma,\ m \in \mathbb{R}$ in Definition \ref{definitionofahmfd}. Suppose that $m>0$ and $\{D_i\}$  is a family of isoperimetric regions with $\mathcal{H}^3_g(D_i)\rightarrow \infty$, we have the following classification:
\begin{enumerate}
  \item $\{D_i\}$ is an exhuastion of $(M,g)$; or
  \item there exists a subsequence of $\{\Sigma_i=\partial D_i\}$ converging to properly, strongly stable, noncompactly complete hypersurface, each connected component $\mathbb{S}$ of which is a constant mean curvature surface of $H=2$. Furthermore, $\mathbb{S}$ is conformally diffeomorphic to complex plane $\mathbb{C}$.
\end{enumerate}
Here, $\mathcal{H}^3_g(~,~ )$ denotes the Hausdroff measure in $(M, g)$ with respect to metric $g$.

\end{theo}

\begin{proof}

If $D_i$ is not an exhaustion, and due to Proposition \ref{nodrift}, we have for  a fixed compact  $E$,
$$E\cap D_i\neq\emptyset,\ \ E\nsubseteq D_i,\  \mbox{for any }\ i .$$
Due to $\mathcal{H}^3_g(D_i)\rightarrow \infty$, we have
\be \label{area-distance}
{\mbox{Area}}(\Sigma_i)\rightarrow \infty \   {\mbox {and}}\  d_g(o,\Sigma_i)\leq L_0
\ee
Here, $L_0$ is a fixed constant.  Hence, by Lemma \ref{squareintergableofA0}, Lemma \ref{hgeq2}, Corollary \ref{strongstability} and Proposition \ref{curvatureestimateisoperisurf}, we obtain the first part in (2).
In the meantime, setting $\phi=1$ in Corollary \ref{strongstability}, we get
$$\label{limitsstrongstability}
\int_{\mathbb{S}}(Ric(v)+\|A\|^2) d\sigma\leq 0,
$$
Due to the Gauss equation, we see that
$$
Ric(v)+\|A\|^2 =\frac{R}2 -K +\frac12 \|\AA\|^2 +\frac{3H^2}4\geq -K.
$$
Hence, we obtain

$$
\int_{\mathbb{S}}K \geq 0.
$$

Together with Lemma \ref{finittotalcurvature} and Huber's theorem (\cite{H}), we see that the conformal type is  complex plane $\mathbb{C}$ or cylinder. If  $\mathbb{S}$ is conformally equivalent to a cylinder, we have
$$
\int_{\mathbb{S}}K = 0.
$$
Due to Theorem \ref{rigidity}, we see that $(M^3, g)$ is isometric to $\mathbb{H}^3$ which contradicts with the assumption in this theorem. Thus
we have proved that $\mathbb{S}$ is conformally equivalent to complex plane $\mathbb{C}$.

\end{proof}

Now, we prove the following result.

\begin{theo}
Let $(M^3,g)$ be an AH with $R(g)\geq -6$, $h=m\sigma$ and $m>0$. If $\mathcal{H}^3_g(D_i)\rightarrow \infty$ and $\Sigma_i=\partial D_i$ is topological sphere. If in addition,
$$  m_H(\Sigma_i)\leq C, \quad \text{for all}~~ i.$$
then  $\{D_i\}$ is an exhaustion of $(M,g)$.
  \end{theo}

\begin{proof}
We prove the theorem  in following two steps.\\

\textbf{Step1:} {\it Show that for all $\epsilon >0$,exists $N$, for all} $i\geq N$, $\int_{\Sigma_i}|\AA|^2 \leq\epsilon$.
In fact, by Proposition 3.6 in \cite{OC}
$$\int_{\Sigma_i}(R_g+6+|\AA|^2) d\mu_i\leq \frac{3}{2}A(\Sigma)^{-\frac{1}{2}}(16\pi)^{\frac{3}{2}}m_H(\Sigma_i).$$
Hence, for all $\epsilon >0$,exists $N$, for all $i\geq N$
$$\int_{\Sigma_i}|\AA|^2 \leq\epsilon.$$

\textbf{Step2:} Suppose $D_i$ is not exhaustion, then by Proposition \ref{nodrift}, all $D_i$ passes through a fixed compact set, and hence
it converges to a limit surface $\mathbb{S}$. Note that on $\Sigma_i$,
$$K_i=\frac{H_i^2-4}{4}-\frac{|\AA_i|^2}{2}+O(e^{-3\rho}).$$
we integrate on region $\Sigma_i\setminus B(\rho)$, and get
$$\int_{\Sigma_i\setminus B(\rho)}K_i=\int_{\Sigma_i\setminus B(\rho)}\frac{H_i^2-4}{4}-\int_{\Sigma_i\setminus B(\rho)}\frac{|\AA_i|^2}{2}+\int_{\Sigma_i\setminus B(\rho)}O(\exp(-3\rho)).$$
Then, we have the following estimate
$$\int_{\Sigma_i\setminus B(\rho)}\frac{H_i^2-4}{4}=4\pi-\int_{\Sigma_i\cap B(\rho)}\frac{H_i^2-4}{4}+O(\frac{1}{|\Sigma_i|^{\frac{1}{2}}}).$$
$$\int_{\Sigma_i\setminus B(\rho)}\frac{|\AA_i|^2}{2}\leq \frac{\epsilon}{3}.$$
$$\int_{\Sigma_i\setminus B(\rho)}O(\exp(-3\rho)) \leq \frac{\epsilon}{3}.$$
Then, $$\int_{\Sigma_i\setminus B(\rho)}K=4\pi+O(\epsilon).$$
implies
$$\int_{\mathbb{S}\cap B(\rho)}K=O(\epsilon).$$
Thus,
$$\int_{\mathbb{S}}K=0.$$
By  Theorem \ref{rigidity}, we have that $(M^3,g)$ is isometric to $\mathbb{H}^3$ which contradicts with the assumption condition that $m>0$. Hence, $D_i$ is an exhaustion.
\end{proof}

In the following, we begin to prove the last theorem in the article,

\begin{theo}
Let $(M^3,g)$ be an AH with $R(g)\geq -6$. For any exhausting isoperimetric domains $\{D_i\}$ ,
we have
$$\lim_{i\rightarrow\infty}(A_g(v_i)-A_{{\mathbb{H}}}(v_i))=-2V(M,g).$$

\end{theo}

\begin{proof}
Combining with lemma \ref{largev}, we just need to prove:
$$\lim_{v\rightarrow\infty}A_g(v)-A_{{\mathbb{H}}}(v)\geq-2V(M,g).$$
Because $\mathcal{D}_i$ is an exhaustion of $M$ and $(M,g)$ is asymptotically hyperbolic, we can easily get $$\lim\limits_{i\rightarrow\infty}(\mathcal{H}^3_{g}(D_i)-\mathcal{H}^3_{g_{\mathbb{H}}}(D_i))=V(M,g),\  \mathcal{H}^2_g(\Sigma_i)-\mathcal{H}^2_{\mathbb{H}}(\Sigma_i)=o(1),\ \ i\rightarrow\infty $$
Choosing $\rho_v$ satisfying with $A_{{\mathbb{H}}}(v)=4\pi \sinh^2\rho_v$, we obtain
\begin{eqnarray*}
v&=&\int_{0}^{\rho_v}4\pi \sinh^{2}\rho \  d\rho=2\pi \sinh^{2}\rho_v+\pi-2\pi\rho_v-\pi e^{-2\rho_v}.
\end{eqnarray*}
Set $\bar{v}_i=\mbox{vol}_{g_{\mathbb{H}}}(D_i)$. Then, $$A_{{\mathbb{H}}}(\bar{v}_i)\leq \mathcal{H}^2_{\mathbb{H}}(\Sigma_i)\  \mbox{and} \ \lim\limits_ {i\rightarrow\infty}(v_i-\bar{v}_i)=V(M,g).$$
Hence,
\begin{eqnarray*}
\lim_{i\rightarrow\infty}(2\pi \sinh^{2}\rho_{v_i}-2\pi\rho_{v_i}-\pi e^{-2\rho_{v_i}}-2\pi \sinh^{2}\rho_{\bar{v}_i}+2\pi\rho_{\bar{v}_i}+\pi e^{-2\rho_{\bar{v}_i}})=V(M,g).
\end{eqnarray*}
 Notice that $i\rightarrow\infty$ implies  $$v_i\rightarrow\infty,\  \bar{v}_i\rightarrow\infty,\   \rho_{v_i}\rightarrow\infty,\  \rho_{\bar{v}_i}\rightarrow\infty.$$
 So,
\begin{eqnarray*}
\lim_{i\rightarrow\infty}(2\pi \sinh^{2}\rho_{v_i}-2\pi\rho_{v_i}-2\pi \sinh^{2}\rho_{\bar{v}_i}+2\pi\rho_{\bar{v}_i} )=V(M,g);
\end{eqnarray*}
i.e.,
\begin{eqnarray*}
2\pi e^{2\rho_{\bar{v}_i}} (e^{2(\rho_{v_i}-\rho_{\bar{v}_i})}-1)-2\pi(\rho_{v_i}-\rho_{\bar{v}_i})\rightarrow V(M,g),
\end{eqnarray*}
so
$$e^{2(\rho_{v_i}-\rho_{\bar{v}_i})}-1=O(e^{-2\rho_{\bar{v}_i}}),$$
and
$$\rho_{v_i}-\rho_{\bar{v}_i}=O(e^{-2\rho_{\bar{v}_i}}).$$
  Therefore,
  \begin{eqnarray*}
\lim_{i\rightarrow\infty}(2\pi \sinh^{2}\rho_{v_i}-2\pi \sinh^{2}\rho_{\bar{v}_i})=V(M,g).
\end{eqnarray*}
and
\begin{eqnarray*}
&&A_g(v_i)-A_{{\mathbb{H}}}(v_i)\\
&&=A_g(v_i)-A_{{\mathbb{H}}}(\bar{v}_i)+A_{{\mathbb{H}}}(\bar{v}_i)-A_{{\mathbb{H}}}(v_i)\\
&&\geq \mathcal{H}^2_g(\Sigma_i)-\mathcal{H}^2_{\mathbb{H}}(\Sigma_i)+4\pi \sinh^2\rho_{\bar{v}_i}-4\pi \sinh^2\rho_{v_i}\\
&&=-2V(M,g)+o(1).
\end{eqnarray*}
Thus, we finish proof of the theorem.

\end{proof}

\end{document}